\documentclass[reqno,10pt, centertags,draft]{amsart}
\usepackage{amsmath,amsthm,amscd,amssymb,latexsym,upref}
\usepackage{latexsym}
\usepackage{lscape}
\usepackage{url}

\renewcommand{\i}{\hbox{\rm i}}

\newcommand{\bbN}{{\mathbb{N}}}
\newcommand{\bbR}{{\mathbb{R}}}

\newcommand{\bbZ}{{\mathbb{Z}}}
\newcommand{\bbC}{{\mathbb{C}}}

\newcommand{\bbT}{{\mathbb{T}}}

\newcommand{\cB}{{\mathcal B}}

\newcommand{\cE}{{\mathcal E}}
\newcommand{\cF}{{\mathcal F}}
\newcommand{\cG}{{\mathcal G}}

\newcommand{\cT}{{\mathcal T}}

\newcommand{\cW}{{\mathcal W}}

\newcommand{\lb}{\label}

\newcommand{\bu}{{\mathbf u}}
\newcommand{\bv}{{\mathbf v}}
\newcommand{\bk}{{\mathbf k}}

\newcommand{\bY}{{\mathbf Y}}
\newcommand{\bp}{{\mathbf p}}
\newcommand{\bq}{{\mathbf q}}
\newcommand{\bx}{{\mathbf x}}
\newcommand{\by}{{\mathbf y}}

\newcommand{\bz}{{\mathbf z}}

\newcommand{\spec}{\operatorname{Spec}}

\newcommand{\ran}{\operatorname{ran}}

\renewcommand{\div}{\operatorname{div}}
\renewcommand{\det}{\operatorname{det}}

\newcommand{\re}{\operatorname{Re}}

\newcommand{\curl}{\operatorname{curl}}

\newcommand{\diag}{\operatorname{diag}}

\theoremstyle{plain}
\newtheorem{theorem}{Theorem}[section]

\newtheorem{lemma}[theorem]{Lemma}
\newtheorem{corollary}[theorem]{Corollary}

\theoremstyle{definition}

\newtheorem{remark}[theorem]{Remark}

\begin{document}
\allowdisplaybreaks

\title[Characteristic determinants]{Characteristic determinants for a second order difference equation on the half-line arising in hydrodynamics}

\author[Y. Latushkin]{Yuri Latushkin}
\thanks{Y.L was supported by the NSF grant DMS-2106157, and would like to
thank the Courant Institute of Mathematical Sciences at NYU and
especially Prof.\ Lai-Sang Young for their hospitality.}
\address{University of Missouri, Columbia, MO 65211, USA}
\email{latushkiny@missouri.edu}

\author[S. Vasudevan]{Shibi Vasudevan}
\address{Krea University, Sri City, Andhra Pradesh, 517646, India}
\email{shibi.vasudevan@krea.edu.in}
\date{\today}
\dedicatory{To Yuriy Karlovich on his seventy fifth birthday}


%
%

\begin{abstract}
We study the point spectrum of a second order difference operator with complex potential on the half-line via Fredholm determinants of the corresponding Birman-Schwinger operator pencils, the Evans and the Jost functions. An application is given to instability of a generalization of the Kolmogorov flow for the Euler equation of ideal fluid on the two dimensional torus.
\end{abstract}

\maketitle



\section{Introduction and main results} In this paper we continue the work began in \cite{LV24} and study the eigenvalues of the following boundary value problem for the second order  asymptotically autonomous difference equation on the half-line $\bbZ_+=\{0,1,\dots\}$,
\begin{eqnarray}\label{main}
z_{n-1}-z_{n+1}+(b_nc_n)z_n&=&\lambda z_n, \quad n\ge0,\\
\label{main-bc}
z_{-1}&=&0,
\end{eqnarray}
where  $\lambda\in\bbC$ is  the spectral parameter,
 \begin{equation}\label{condbc}
 (b_n)_{n\ge0}\in\ell^2(\bbZ_+;\bbC) \text{ and } (c_n)_{n\ge0}\in\ell^2(\bbZ_+;\bbC)
 \end{equation} are two given complex valued sequences that, in general, may depend on $\lambda$ holomorphically. We are seeking to characterize the values of the spectral parameter such that \eqref{main}--\eqref{main-bc} has a nontrivial solution $\bz=(z_n)_{n\ge0}\in\ell^2(\bbZ_+;\bbC)$. This question is important in stability issues for special steady state solutions of the two-dimensional Euler equation, the so called {\em generalized Kolmogorov flow}, and our main application is a result on its instability.
 Specifically, in the current paper we define (and prove that they are being equal) four functions of the spectral parameter whose zeros are the eigenvalues of \eqref{main}--\eqref{main-bc}. We call the functions {\em characteristic determinants}.
 Our overall strategy is as in \cite{LV24} where the full line case has been considered, but the treatment of the important half-line case is technically more challenging;
  in particular, in the current paper we have to prove from the outset for the half-line case  analogs of some results obtained in \cite{CL07} for the line case and essentially used in \cite{LV24}. By reflection, analogous results hold for the equations on the negative half-line.
 
 An important particular case  of \eqref{main}--\eqref{main-bc} is the following eigenvalue problem for a difference equation arising in stability analysis of the generalized Kolmogorov flow of the Euler equation of ideal fluids on 2D torus, as seen 
 below and considered in \cite{DLMVW20} and \cite{LV24},
\begin{equation}\label{eulerev1}
z_{n-1}-z_{n+1}={\lambda}z_n/{\rho_n}, \quad n\ge0,\quad
z_{-1}=0,
\end{equation}
where $(\rho_n)_{n\ge0}$ is a given sequence satisfying the following conditions:
\begin{equation}\label{condrho}
\rho_0<0,\, \rho_n\in(0,1), \, n\ge1, \text{ and $\rho_n=1+O(1/n^2)$ as $|n|\to\infty$.}
\end{equation}
 The problem \eqref{eulerev1} is reduced to \eqref{main}--\eqref{main-bc} by setting 
 \begin{equation}\label{defbc}
 b_n=-\lambda\sqrt{1-\rho_n}/\rho_n \text{ and } c_n=\sqrt{1-\rho_n}.\end{equation}
For \eqref{eulerev1} we use continued fractions to define yet another, fifth function, also proven to be equal to the previous four characteristic determinants, whose zeros are the eigenvalues. This result, in turn, yields instability of the generalized Kolmogorov flow.

It is convenient to re-write \eqref{main}--\eqref{main-bc} as 
\begin{equation}\label{ev3}
\big(S-S^{-1}+\diag_{n\in\bbZ_+}\{b_nc_n\}\big)\bz=\lambda\bz,
\end{equation}
where we denote by $S:(z_n)_{n\ge0}\mapsto (0,z_0,z_1,\dots)$ the right shift operator on $\ell^2(\bbZ_+;\bbC)$ and by $S^{-1}=S^*:\bz\mapsto(z_1,z_2,\dots)$ the left shift operator  such that $S^{-1}S=I_{\ell^2(\bbZ_+;\bbC)}$, the identity operator, while $\ran(SS^{-1})=({\rm span } \{(1,0,\dots)\})^\bot$. By the spectral mapping theorem $\spec(S-S^{-1})=[-2\i, 2\i]$ since $\spec(S)=\{\lambda\in\bbC: |\lambda|\le1\}$, and so throughout the paper we assume that $\lambda\notin[-2\i, 2\i]$ thus looking for the {\em isolated eigenvalues} of  \eqref{main}--\eqref{main-bc}.

Letting $y_n=\begin{bmatrix}z_n\\z_{n-1}\end{bmatrix}\in\bbC^2$, the problem \eqref{main}--\eqref{main-bc} is equivalent to the  following boundary value problem for the first order $(2\times 2)$-system of difference equations,
\begin{eqnarray}\label{evsys1}
y_{n+1}&=&A_n^{\times}y_n,\, n\ge0,\\
\label{evsys1-bc}
y_0&\in&\ran ( Q_+).
\end{eqnarray}
Here and throughout the paper we use the following notations,
\begin{eqnarray}\label{Anot1}
A_n^{\times}&=&A+B_nC_n:=\begin{bmatrix}b_nc_n-\lambda\,&1\\ \, 1&\, 0\end{bmatrix}\in\bbC^{2\times 2},\, n\ge0.\\
A=A(\lambda)&=&\begin{bmatrix}-\lambda&1\\1&0\end{bmatrix}, Q_+=\begin{bmatrix}1&0\\0&0\end{bmatrix},  Q_-=\begin{bmatrix}0&0\\0&1\end{bmatrix}, B_n=b_nQ_+, C_n=c_nQ_+.\label{Anot}
\end{eqnarray}
 The eigenvalues of the matrix $A(\lambda)$ solve the quadratic equation $\mu^2+\lambda\mu-1=0$, and so our standing assumption $\lambda\notin[-2{\rm i},2{\rm i}]$ is equivalent to the fact that the eigenvalues of $A(\lambda)$ are off the unit circle. 
We let $\mu_+=\mu_+(\lambda)$ and $\mu_-=\mu_-(\lambda)$ denote the roots of the equation 
$\mu^2+\lambda\mu-1=0$ satisfying the inequalities
\begin{equation}\label{muineq}
|\mu_+(\lambda)|<1<|\mu_-(\lambda)|
\end{equation}
and denote by $P_\pm=P_\pm(\lambda)$ the spectral projections for $A(\lambda)$ in $\bbC^2$ such that 
$\spec(A(\lambda)\big|_{\ran P_\pm})=\{\mu_\pm(\lambda)\}$. Finally, we let $R_+=R_+(\lambda)$ denote the projection in $\bbC^2$ onto $\ran P_+$ parallel to $\ran Q_+$, and set $R_-=I_{2\times 2}-R_+$. 

The choice of the projection $R_+$ is important for the half-line case. Indeed, the constant coefficient difference equation $y_{n+1}=Ay_n$ has exponential dichotomy on $\bbZ_+$ with the dichotomy projection $P_+$ whose range is the uniquely determined subspace of the initial values of the bounded solutions to the equation. Unlike the full line case, the exponential dichotomy on $\bbZ_+$ is not unique, but the only requirement on the dichotomy projection is that its range must be equal to the subspace $\ran P_+$. The choice of $R_+$ whose kernel is $\ran Q_+$ as the dichotomy projection will allow us to satisfy the boundary condition in 
\eqref{main-bc}, see, e.g., formula \eqref{invform} below.

We now proceed to define our first characteristic determinant associated with \eqref{main}--\eqref{main-bc}. We consider the following Birman-Schwinger-type pencil 
of operators acting in $\ell^2(\bbZ_+;\bbC)$,
\begin{equation}\label{defK}
K_\lambda^+=-\diag_{n\in\bbZ_+}\{c_n\}(S-S^{-1}-\lambda)^{-1}\diag_{n\in\bbZ_+}\{b_n\},
\end{equation}
analogous to $K_\lambda$ studied in \cite{LV24} for the full line.
 By \eqref{condbc}, the operator $K_\lambda^+$ is of trace class and so we may define our first  characteristic determinant $\det(I-K_\lambda^+)$.

Next, we re-write \eqref{evsys1} as 
$\big(S^{-1}-\diag_{n\in\bbZ_+}\{A_n^{\times}\}\big)\by=0$ for $\by=(y_n)_{n\in\bbZ_+}$ where we continue to denote by $S^{-1}$ 
the shift acting in the space of vector valued sequences. We stress that the operator $S^{-1}-\diag_{n\in\bbZ_+}\{A_n^{\times}\}$ in \eqref{evsys1}--\eqref{evsys1-bc} acts from the subspace \begin{equation}\label{defellQ}
\ell^2_{\ran(Q_+)}(\bbZ_+;\bbC^2):=\{(y_n)_{n\ge0}\in\ell^2(\bbZ_+;\bbC^2): y_0\in\ran(Q_+)\}\end{equation} into $\ell^2(\bbZ_+;\bbC^2)$ and refer to \cite{BAGK}  for  results on equivalence of invertibility of the operator 
and dichotomy of the difference equation \eqref{evsys1} on the half-line. We  introduce the Birman-Schwinger-type pencil 
of operators acting in $\ell^2(\bbZ_+; \bbC^2)$,
\begin{equation}\label{defT}
\cT_\lambda^+=\diag_{n\in\bbZ_+}\{C_n\}\big(S^{-1}-\diag_{n\in\bbZ_+}\{A\}\big)^{-1}\diag_{n\in\bbZ_+}\{B_n\},
\end{equation}
analogous to $\cT_\lambda$ studied in \cite{LV24} for the full line case.
 The operator $\cT_\lambda^+$ is of trace class by \eqref{condbc} and we define our second characteristic determinant $\det(I-\cT_\lambda^+)$.

We now consider the $(2\times 2)$-{\em matrix valued Jost solution} $\bY^+=\bY^+(\lambda)=(Y^+_n)_{n\ge0}$, $Y^+_n\in\bbC^{2\times 2}$, cf.\ \cite{CL07, GLM07,LV24}, defined as the solution to the difference equation \eqref{evsys1} (without the boundary condition) satisfying
\begin{equation}\label{mvj}
\big\|(\mu_+)^{-n}\big(Y^+_n-A^nR_+\big)\|_{\bbC^{2\times 2}}\to0 \text{ as $n\to+\infty$ and $Y^+_0=Y^+_0R_+$.}
\end{equation}
As we will prove below, this solution is unique. Also, $Y_n^+=Y_n^+R_+$ for all $n\ge0$.
Using the projection $R_-=R_-(\lambda)$ onto $\ran Q_+$ parallel to $\ran P_+(\lambda)$, analogously to $\cE(\lambda)$ in \cite{LV24} for the full line case, we introduce our third characteristic determinant, the {\em Evans function}, by the formula
\begin{equation}\label{defcE}
\cE^+(\lambda)=\det\big(Y_0^+(\lambda)+R_-(\lambda)\big).
\end{equation}

We call a solution $\bz^+=(z^+_n)_{n\ge-1}$ of the second order difference equation \eqref{main}  (without the boundary condition) the {\em Jost solution,}  cf.\ \cite{CS}, provided 
\begin{equation}\label{zpm-asymp}
(\mu_+)^{-n}z^+_n-1\to0 \text{ as $n\to+\infty$},
\end{equation}
and denote by $\bz^{\rm r}=(z_n^{\rm r})_{n\ge-1}$ the {\em regular} solution of the boundary value problem \eqref{main}--\eqref{main-bc} satisfying the additional boundary condition $z_0^{\rm r}=1$ (which is of course unique but not necessarily bounded at infinity).
As in the case of the discrete Schr\"odinger equations, the Jost solution $\bz^+=\bz^+(\lambda)$ is unique. Introducing the notation $\cW(\bu,\bv)_n=(-1)^{n-1}\big(u_{n-1}v_n-u_nv_{n-1}\big)$, $n\ge0$, for the Wronskian of any two sequences $\bu=(u_n)_{n\ge-1}$ and $\bv=(v_n)_{n\ge-1}$, we note that the Wronskian of the solutions to  \eqref{main}  is $n$-independent, and that $\cW(\bz^+,\bz^{\rm r})_0=z^+_{-1}(\lambda)$. We define our fourth characteristic determinant, the {\em Jost function}, analogously to $\cF$ from \cite{LV24} for the full line case,
\begin{equation}\label{defF}
\cF^+(\lambda)=\mu_+(\lambda)\cW(\bz^+(\lambda),\bz^{\rm r}(\lambda))_0,
\end{equation}
where the scaling factor $\mu_+(\lambda)$ is chosen such that $\cF^+(\lambda)=1$ provided $b_nc_n=0$ for $n\ge0$ when $z_n^+=\mu_+^n$ for $n\ge-1$.

Our fifth characteristic determinant will be defined only for the difference equation \eqref{eulerev1} assuming \eqref{condrho} and uses continued fractions; we introduce the notation 
\begin{align}\lb{deffg}
g_+(\lambda)=
\cfrac{1}{
\frac{\lambda}{\rho_{1}}+\cfrac{1}{
\frac{\lambda}{\rho_{2}}+\dots}}.\,\,
\end{align}
The continued fraction converges for $\re(\lambda)>0$ and $|\arg(\lambda)|\le \pi/2-\delta$ for any $\delta\in(0,\pi/2)$ by the classical Van Vleck Theorem \cite[Theorem 4.29]{JT} and we may now define our fifth function of interest by the formula, cf.\  $\cG$ from \cite{LV24},
\begin{equation}\label{defG}
\cG^+(\lambda)= \mu_+(\lambda)\, z_0^+(\lambda)\,\big(g_+(\lambda)+\lambda/\rho_0\big),
\end{equation}
where $z^+_0(\lambda)$ is the $0$-th entry of the Jost solution, $\mu_+(\lambda)$ is the eigenvalue of $A(\lambda)$ inside of the unit disk, and  $g_+(\lambda)$ is defined in \eqref{deffg},
and we assume that $\re(\lambda)>0$.

We are ready to formulate the main results of this paper.

\begin{theorem}\label{thm:main}
 Assume $\lambda\notin[-2{\rm i}, 2 {\rm i}]$ and that the sequences in \eqref{condbc} depend on $\lambda$ holomorphically. The functions introduced in \eqref{defK}, \eqref{defT}, \eqref{defcE} and \eqref{defF} are  holomorphic in $\lambda$ and equal, 
\begin{equation}\label{bigfive}
\det(I-K_\lambda^+)=\det(I-\cT_\lambda^+)=\cE^+(\lambda)=\cF^+(\lambda)).
\end{equation}
As a result,  $\lambda$ is a simple discrete eigenvalue of \eqref{main}--\eqref{main-bc} if and only if $\lambda$ is a zero of each of the functions in \eqref{bigfive}.\end{theorem}

\begin{corollary}\label{cor:main}
Assume $\re(\lambda)>0$ and \eqref{condrho}, and  consider the difference equation \eqref{eulerev1}. Define $(b_n)$ and $(c_n)$ by \eqref{defbc} and use the sequences to construct all four characteristic determinants in \eqref{bigfive}.
Then $\cG^+(\lambda)$ defined in \eqref{defG} via the continued fraction is equal to the functions in \eqref{bigfive}.
\end{corollary}

We refer to \cite{LV24} for a detailed discussion of the literature related to the results, and mention here only the following. Regarding the Birman-Schwinger operator pencils see \cite{BtEG,CGLNSZ}. The papers most relevant to the current setup are \cite{CL07,GLM07} and \cite{LV18}. 
We are not aware of any literature on the Birman-Schwinger type pencils specific for the first order  systems \eqref{main} on the half-line. For the full line case they have been studied in \cite{CL07} which is a companion paper to \cite{GLM07} dealing with differential equations where, among many other things, 
 the first equality in \eqref{bigfive} was proved for the case of the Schr\"odinger differential operators. We are not aware of results of this sort for the difference equations of type \eqref{main} on the half-line. This equation is of course a particular case of eigenvalue problems for three diagonal (Jacobi) matrices, cf.\  \cite{LG21,JN98} and also \cite{BS21,EG05,FLW23,G22,KLS22,SS13}. The dichotomies on the half-line were studied in great detail, see \cite{BAG, BAGK93,BAGK} and the references therein.
Regarding  the Evans function we refer to \cite{KP,S} where one can find many other sources. The most relevant papers are again \cite{CL07,GLM07}, where, in particular, the second equality in \eqref{bigfive} has been proved for quite general first order systems of difference and differential equations but on the full line, and so the results in the current paper for the half-line case seem to be new. The Evans function is a relatively new topic in the study of the 2D-Euler-related difference equations as in \eqref{eulerev1}; we are aware of only
\cite{DM24, LV24}. The Jost solutions  are classical  \cite{CS} and so is the equality of $\det(I-K_\lambda^+)$ and $\cF^+$ in \eqref{bigfive} for the Schr\"odinger case, see \cite{JP51,Ne72} and generalizations in  \cite{GMZ07,GN15}; however, it is quite possible that the use of the Jost solutions as well as the equality of $\cF^+$ and $\cE^+$ and $\det(I-K_\lambda^+)$ in the context of \eqref{main}, see \cite{LV24}, appear to be new for the half-line case. Regarding continued fractions in this context see
\cite{DLMVW20,FH98,JT,MS, SV21}.

We now briefly discuss connections to the Euler equation again referring to \cite{DLMVW20,LV24} for more details and
concentrating on just one particularly important case of stability of the so-called unidirectional, or generalized Kolmogorov, flow for the Euler equations of ideal fluid on the two dimensional torus $\bbT^2$,
\begin{equation}\label{e}
\partial_{t} \Omega + U \cdot \nabla \Omega = 0,
\, \div U=0, \, \Omega=\curl U, \,  \bx=(x_1,x_2)\in\bbT^2,
\end{equation}
where the two-dimensional vector $U=U(\bx)$ is the velocity and the scalar $\Omega$ is the vorticity of the fluid. 
The unidirectional (or generalized Kolmogorov) flow is the steady state solution to the Euler equations on $\mathbb T^{2}$ of the form
\begin{equation}\label{uni}
\Omega^{0}(\mathbf x) = \alpha e^{i \mathbf p \cdot \mathbf x}/2+ \alpha e^{-i \mathbf p \cdot \mathbf x}/2 = \alpha \cos (\mathbf p \cdot \mathbf x),
\end{equation}
where $\mathbf p \in \mathbb Z^{2} \backslash \{0\}$ is a given vector
 and $\alpha\in\bbR$.
In particular, $\bp=(m,0)\in\bbZ^2$ for $m\in\bbN$ corresponds to the classical Kolmogorov flow.
It is thus a classical problem to study (linear) stability of the flow given by \eqref{uni}. To this end, using Fourier series $\Omega(\bx)=\sum_{\bk\in\bbZ^2\setminus\{0\}}\omega_\bk {\rm e}^{{\rm i} \bk\cdot\bx}$ for vorticity, we rewrite \eqref{e} as a system of nonlinear equations for $\omega_\bk$, $\bk\in\bbZ^2$ as in \cite{DLMVW20,LLS,LV24,L}. Linearizing this system about the unidirectional flow, one obtains the following operator in $\ell^2(\mathbb Z^2;\bbC)$,
\begin{align}
L&: (\omega_\mathbf k)_{\mathbf k\in\mathbb Z^2}\mapsto
\big(\alpha\beta(\mathbf p,\mathbf k-\mathbf p)\omega_{\mathbf k-\mathbf p}-
\alpha\beta(\mathbf p,\mathbf k+\mathbf p) \omega_{\mathbf k+\mathbf p}\big)_{\mathbf k\in\mathbb Z^2},\label{dfnLB}
\end{align}
where the coefficients $\beta(\mathbf p,\mathbf q)$ for $\mathbf p=(p_1,p_2)$, $\mathbf q=(q_1,q_2)\in\mathbb Z^2$ are defined as 
\begin{equation}\label{dfnalpha}
\beta(\mathbf p,\mathbf q)=\frac{1}{2}\big(\|\mathbf q\|^{-2}-\|\mathbf p\|^{-2}\big)(\mathbf p\wedge\mathbf q),\,\text{ with } \mathbf p\wedge\mathbf q:=\det\left[\begin{smallmatrix}p_1&q_1\\ p_2&q_2\end{smallmatrix}\right],
\end{equation}
for 
$\mathbf p\neq0,\mathbf q\neq0$, and $\beta(\mathbf p,\mathbf q)=0$ otherwise.  
The flow \eqref{uni} is called (linearly) unstable if the operator $L$ has
 nonimaginary spectrum. We refer to \cite{LV19, ZL04} for results on stability and instability for the 2D Euler equations and related models.
 
To study the spectrum of $L$, (see the discussion in \cite[pp. 2054-2057]{DLMVW20}; see also \cite{LLS,LV24}) we decompose this operator 
 into a sum of operators $L_{\mathbf q}$, $\mathbf q\in\mathbb Z^2$, acting in the space $\ell^2(\mathbb Z;\bbC)$, by ``slicing'' the grid $\bbZ^2$ along lines parallel to $\bp$ such that 
 $\spec(L)=\cup_\bq\spec(L_\bq)$, where 
\begin{equation}
L_{\mathbf q}: (w_n)\mapsto
\big(\alpha\beta(\mathbf p,\mathbf q+(n-1)\mathbf p) w_{n-1}-
\alpha\beta(\mathbf p,\mathbf q+(n+1)\mathbf p) w_{n+1}\big),\, n\in\mathbb Z,\label{dfnLB1}
\end{equation}
and for $\mathbf k=\mathbf q+n\mathbf p$ from \eqref{dfnLB} we denote $w_n=\omega_{\mathbf q+n\mathbf p}$. In other words,
 $L_\bq=\alpha(S-S^*)\diag_{n\in\bbZ}\{\beta(\bp, \bq+n\bp)\}$.
By \eqref{dfnalpha}, if $\mathbf q\parallel\bp$ then $L_{\mathbf q}=0$ and thus we assume throughout that $\mathbf q\nparallel\mathbf p$. Moreover, since  $L_\bq$ contains a scalar multiple $\alpha\in\bbR$, with no loss of generality we may rescale this operator or, equivalently, will assume the normalization condition
$\alpha(\bq\wedge\bp)\|\bp\|^{-2}/2=1$. 

The operators $L_\bq$ are classified based on the location of the line $B_{\bq}=\{\bq+n\bp: n\in\bbZ\}$ through the point $\bq\in\bbZ^2$ relative to the disc of radius $\|\bp\|$ centered at zero, see, again, the discussion in \cite[pp. 2054-2057]{DLMVW20}; see also \cite{LLS,LV24}. Spectral properties of the operators drastically depend on the location. For instance, if none of the vectors $\bq+n\bp$, $n\in\bbZ$, is located inside the open disc then $L_\bq$ has no unstable eigenvalues \cite{LLS,L}.  In the current paper we consider only the case when 
$\bp$ and $\bq$ are such that 
\begin{equation}\label{condpq}
\|\bq\|<\|\bp\|,\ \|\bq-\bp\|=\|\bp\| \text{ and  $\|\bq+n\bp\|>\|\bp\|$ for all 
$n\in\bbZ\setminus\{-1,0\}$;}
\end{equation} 
 a typical example of this is $\bp=(3,1)$, $\bq=(2,-2)$. This assumption corresponds to the case $I_-$ described in \cite{DLMVW20}, and we refer to this paper for a discussion regarding other possible cases. The case $I_+$ when $-1$ in \eqref{condpq} is replaced by $+1$ can be treated similarly to $I_-$ while the case $I_0$ when  $\|\bq+n\bp\|>\|\bp\|$ for all $n\neq0$ has been considered in \cite{LV24}. Geometrically, the case $I_-$ and $I_+$ occur when one of the points in $B_\bq$ is inside of the open disc of radius $\|\bp\|$ and one of the points in $B_\bq$ is on the boundary of the disk while in case $I_0$ one point of $B_\bq$ is inside of the open disk while all others are outside of the closed disk. An interesting open question is to describe the spectrum of $L_\bq$ in the case $II$ when two points of $B_\bq$ are located inside of the open disk of radius $\|\bp\|$ centered at zero.

Assuming that we are in the case $I_-$, that is, that \eqref{condpq} holds, we introduce the sequence 
$\rho_n=1-\|\bp\|^2\|\bq+n\bp\|^{-2}$, $n\in\bbZ$, such that the operator $L_\bq$ in \eqref{dfnLB1} reads  $L_\bq=(S-S^*)\diag_{n\in\bbZ}\{\rho_n\}$, see \eqref{dfnalpha}. 
The eigenvalue equation for $L_\bq$ is
\begin{equation}\label{weq}
\rho_{n-1}w_{n-1}-\rho_{n+1}w_{n+1}=\lambda w_n, \, n\in\bbZ,\, \text{ with }\,(w_n)_{n\in\bbZ}\in\ell^2(\bbZ;\bbC).
\end{equation}
By \eqref{condpq} we have $\rho_{-1}=0$, $\rho_0<0$,  and $\rho_n\in(0,1)$ for all $n\in\bbZ\setminus\{0,-1\}$. In particular, the sequence $(\rho_n)_{n\ge0}$ satisfies \eqref{condrho}. Now Theorem \ref{thm:main} and Corollary \ref{cor:main} can be applied because the solution of \eqref{weq} must be supported on $\bbZ_+$ as seen in Lemma \ref{lem:hl} given in the next section.
Theorem \ref{thm:main} implies the following.
  \begin{corollary}\label{cor:main:eul}
 Assume that a given $\bp\in\bbZ^2$ is such that there exists a $\bq\in\bbZ^2$ satisfying
 \eqref{condpq}. Then the eigenvalues of $L_\bq$ with positive real parts are in one-to-one  correspondence with  zeros of each of the five functions in \eqref{bigfive} and \eqref{defG}. Moreover,  the operator $L_\bq$ from \eqref{dfnLB1}, and thus $L$ from \eqref{dfnLB}, has a positive eigenvalue. As a result, the unidirectional flow \eqref{uni} is linearly unstable.
 \end{corollary}
The instability of the unidirectional flow in the current setting has been established in \cite[Theorem 2.9]{DLMVW20}. Nevertheless, the first assertion in Corollary \ref{cor:main:eul}  is an improvement of the part of \cite[Theorem 2.9]{DLMVW20} where the correspondence was established between only the {\em positive} roots of the function $g_+(\lambda)+\lambda/\rho_0$ and the {\em positive} eigenvalues of $L_\bq$ but under the additional assumption that the respective eigensequences satisfy some special property \cite[Property 2.8]{DLMVW20}. By applying Theorem \ref{thm:main} in the proof of Corollary \ref{cor:main:eul} we were able to show that this assumption is redundant.
 
We conclude this section with references on the literature on stability of unidirectional flows. This topic is quite classical and well-studied, and we refer to \cite{BW,BFY99,FZ98,FH98,FVY00,MS}. The setup used herein was also used in many papers \cite{WDM1, WDM2, LLS, L, SV21}, but the closest to the current work is \cite{DLMVW20,LV24}. We mention also \cite{LV18} regarding connections to the Birman-Schwinger operators. Finally, connections between the Evans function and the linearization of the 2D Euler equation has been studied in a recent important paper \cite{DM24}.

\section{Proofs}

We begin with several general comments regarding the objects introduced in the previous section. For a start, 
 it is sometimes convenient to diagonalize the matrix $A$ from \eqref{Anot}. To this end we introduce matrices 
\begin{equation}\label{winv1}
W=\begin{bmatrix}\mu_+&\mu_-\\1&1\end{bmatrix},\,
W^{-1}=(\mu_+-\mu_-)^{-1}\begin{bmatrix}1&-\mu_-\\-1&\mu_+\end{bmatrix},\,
\widetilde{A}=\diag\{\mu_+,\mu_-\},
\end{equation}
so that one has, cf.\ \eqref{Anot},
\begin{equation}\label{AA}
W^{-1}AW=\widetilde{A},\, W^{-1}P_\pm W=Q_\pm.
\end{equation}
We will use below the explicit formulas 
\begin{equation}\label{formRpm}
\overfullrule=0pt
R_+=\begin{bmatrix} 0&\mu_+\\0&1 \end{bmatrix},\, 
R_-=\begin{bmatrix} 1&-\mu_+\\0&0 \end{bmatrix},\, 
P_+=(\mu_+-\mu_-)^{-1}\begin{bmatrix} \mu_+&1\\1&-\mu_-\end{bmatrix},\,
A^{-1}=\begin{bmatrix}0&1\\1&\lambda\end{bmatrix}. 
\end{equation}

The operator $S-S^{-1}-\lambda$ is invertible in $\ell^2(\bbZ_+;\bbC)$ if and only if $\lambda\notin [-2\i, 2\i]$ and in this case the inverse operator $\big(S-S^{-1}-\lambda\big)^{-1}$ is given by the infinite matrix $[ a_{nk} ]_{n,k=0}^\infty$ whose entries are defined by the formula
 \begin{equation}\label{INV0}
 a_{nk}=(\mu_+-\mu_-)^{-1}\times\begin{cases}(\mu_-)^{n-k}\big(1-(\frac{\mu_+}{\mu_-})^{n+1} \big)&\text{ for $n\le k$},\\ 
 (\mu_+)^{n-k}\big(1-(\frac{\mu_+}{\mu_-})^{k+1}\big) &\text{ for $n > k$}.
\end{cases}
 \end{equation}
 One can check this either directly by multiplying the respective infinite matrices and using the equation $\mu^2+\lambda\mu-1=0$, or by using \cite[Theorem 1.1]{GF} as follows:
 we factorize $S-S^{-1}-\lambda=\mu_-(I+\mu_+S^{-1})(I+\mu_-^{-1}S)$ and use the Neumann series to obtain the inverse of each factor.

 From now on we abbreviate $A=\diag_{n\in\bbZ_+}\{A\}$. The operator $S^{-1}-A$ acting from $\ell^2_{\ran Q_+}(\bbZ_+;\bbC^2)$ into
 $\ell^2(\bbZ_+;\bbC^2)$, cf.\ \eqref{defellQ} and \cite{BAGK}, is invertible if and only if $\lambda\notin [-2\i, 2\i]$ and in this case the inverse operator $\big(S^{-1}-A\big)^{-1}$ from $\ell^2(\bbZ_+;\bbC^2)$ onto $\ell^2_{\ran Q_+}(\bbZ_+;\bbC^2)$ for $\bu=(u_n)_{n\ge0}\in\ell^2(\bbZ_+;\bbC^2)$ is given by the formula
 \begin{equation}\label{invform}
 (\big(S^{-1}-A\big)^{-1}\bu)_n=-\sum_{k=n}^{+\infty}A^nR_-A^{-(k+1)}u_k
 + \sum_{k=0}^{n-1}A^nR_+A^{-(k+1)}u_k,\, n\ge 0; \end{equation}
 here and below we always set $\sum_{k=0}^{-1}=0$, and so the RHS of \eqref{invform} for $n=0$ is a vector from $\ran R_-=\ran Q_+$ as required in \eqref{defellQ}. Formula \eqref{invform} can be either taken from \cite{BAGK} or checked directly by multiplying $S^{-1}-A$ and the operator in \eqref{invform} and taking into account that $\ran R_+=\ran P_+$ is the set of the initial data for the {\em bounded} on $\bbZ_+$ solutions of the difference equation $y_{n+1}=Ay_n$, the projection $P_+$ is a dichotomy projection for this equation on $\bbZ_+$, and therefore the projection $R_+$ onto $\ran P_+$ parallel to $\ran Q_+$ is also a dichotomy projection since $\ran P_+\oplus \ran Q_+=\bbC^2$, cf.\ \cite{BAG, BAGK93, BAGK}  and also \cite{CL}, \cite{Cop} or \cite[Chapter 4]{DK} for discussions of dichotomies on the half-line.
 
 Formula \eqref{invform} shows that the operator $\cT_\lambda^+$ in \eqref{defT} is an operator with semi-separable kernel, that is, $\cT_\lambda^+$ can be written as an infinite matrix $\cT_\lambda^+=[T_{nk}]_{n,k=0}^{+\infty}$ so that $(\cT_\lambda^+\bu)_n=\sum_{k=0}^{+\infty}T_{nk}u_k$, where
 \begin{equation}\label{defTnk}
 T_{nk}=\begin{cases}
 -C_nA^nR_-A^{-(k+1)}B_k &\text{ for $0\le n\le k$,}\\ 
 -C_nA^nR_+A^{-(k+1)}B_k &\text{ for $0\le k\le n$.}\end{cases}
 \end{equation}
 We refer to \cite[Chapter IX]{GGK90} and \cite{GM04} for a discussion of the operators with semi-separable kernels; the results therein are used below in the proof of the second equality in \eqref{bigfive}.
 
 The matrix-valued Jost solution $\bY^+=(Y^+_n)_{n\ge0}$ of \eqref{evsys1} (with no boundary condition) is obtained as a solution to the following Volterra equation,
 \begin{equation}\label{VE}
 Y_n^+-A^nR_+=-\sum_{k=n}^{+\infty}A^{n-(k+1)}B_kC_kY_k^+,
 \end{equation}
 defined first for $n\ge N$ with $N$ large enough and then extended to all of $\bbZ_+$ as a solution to the difference equation \eqref{evsys1} via $Y_n^+:=(A^\times_n)^{-1}Y^+_{n+1}$
 for $n=0, 1,\dots,N-1$. The sequence $(Y_n^+)_{n\ge0}$ thus defined will satisfy the Volterra equation \eqref{VE} as the following inductive step shows: Suppose we know that \eqref{VE} holds for $n=N$ and that $B_{N-1}C_{N-1}Y^+_{N-1}=Y^+_N-AY^+_{N-1}$. We then use this and \eqref{VE} with $n=N$ in the following calculation yielding  \eqref{VE} for $n=N-1$,
 \begin{eqnarray*}
 A^{N-1}R_+&-&\sum_{k=N-1}^{+\infty}A^{(N-1)-(k+1)}B_kC_kY_k^+
  \\&=&
  A^{-1}\big( A^{N}R_+-\sum_{k=N}^{+\infty}
  A^{N-(k+1)}B_kC_kY_k^+-Y^+_N\big)+Y^+_{N-1}=Y^+_{N-1}.
 \end{eqnarray*}
 \begin{lemma}\label{lem:Volterra} There is a large enough $N$ such that equation \eqref{VE} for $n\ge N$ has a unique solution thus yielding the matrix-valued Jost solution $\bY^+=(Y^+_n)_{n\ge0}$ of \eqref{evsys1} satisfying \eqref{mvj}.
 \end{lemma}
 \begin{proof} We recall that $|\mu_+|<1$ and introduce the space 
 \[\ell^\infty_N:=\big\{ \bu=(u_n)_{n\ge N}:\, \|\bu\|_{\ell^\infty_N}:=\sup\big\{|\mu_+|^{-n}\|u_n\|_{\bbC^2}: n\ge N\big\}<\infty\big\}\]
 of exponentially decaying at infinity $\bbC^2$-valued sequences on $[N,+\infty)\cap\bbZ_+$. Let $T_N$ denote the operator on $\ell^\infty_N$ that appears in the RHS of \eqref{VE}, 
 \begin{equation}\label{defTN}
 (T_N\bu)_n=-\sum_{k=n}^{+\infty}A^{n-(k+1)}B_kC_ku_k,\, \bu=(u_n)_{n\ge N}\in\ell^\infty_N.
 \end{equation}
Denoting  $q_N=\sum_{n=N}^{+\infty}\|B_kC_k\|_{\bbC^{2\times 2}}$, we have
 $q_N\to0$ as $N\to\infty$ by \eqref{condbc} and \eqref{Anot}. 
  Using  \eqref{AA}  we have\footnote{We write $a\lesssim b$ if $a\le cb$ for a constant $c$ independent of any parameters contained in $a$ and $b$.} 
  $\|A^{n-(k+1)}\|_{\bbC^{2\times 2}}\lesssim|\mu_+|^{n-(k+1)}$ for $k\ge n$. Thus, for $n\ge N$,
  \begin{equation}\label{Tineq}
  |\mu_+|^{-n}\|u_n\|_{\bbC^2}\lesssim\sum_{k=n}^{+\infty}|\mu_+|^{-n}|\mu_+|^{n-(k+1)}\|B_kC_k\|_{\bbC^{2\times 2}}\|u_k\|_{\bbC^2}\lesssim q_N\|\bu\|_{\ell^\infty_N}
  \end{equation}
  and the norm of $T_N$ in $\ell^\infty_N$ is dominated by $q_N$ and therefore is less than, say, $1/2$ provided $N$ is large enough. Then $(Y_n^+)_{n\ge N}=(I-T_N)^{-1}(A^nR_+)_{n\ge N}$
   is the unique solution of \eqref{VE} whose columns are 
   in $\ell^\infty_N$ and so for $n\ge N$ we infer
   \begin{eqnarray*}
   |\mu_+|^{-n}\|Y^+_n&-&A^nR_+\|_{\bbC^{2\times 2}}
   \le\|(Y_n^+)_{n\ge N}-(A^nR_+)_{n\ge N}\|_{\ell^\infty_N}=\|T_N(Y_n^+)_{n\ge N}\|_{\ell^\infty_N}\\
   &\lesssim & q_N\|(Y_n^+)_{n\ge N}\|_{\ell^\infty_N}=q_N\|(I-T_N)^{-1}(A^nR_+)_{n\ge N}\|_{\ell^\infty_N}\\&\lesssim& q_N\|(A^nR_+)_{n\ge N}\|_{\ell^\infty_N}\lesssim q_N
   \end{eqnarray*}
   because $\|(A^nR_+)_{n\ge N}\|_{\ell^\infty_N}\lesssim 1$ since $\ran R_+=\ran P_+$.
   This yields the first assertion in \eqref{mvj} while the second follows by multiplying 
   \eqref{VE} by $R_+$ from the right and using the uniqueness of the solution $(Y_n^+)_{n\ge N}=(Y^+_nR_+)_{n\ge N}$.
 \end{proof}
 
  \begin{remark}\label{rem:Y+}
 Assertion $Y^+_n=Y^+_nR_+$, cf.\ \eqref{mvj}, and formula \eqref{winv1} for $R_+$ show that the first column of the matrix $Y^+_n$ is zero. Since $(Y^+_n)_{n\ge0}$ solves the difference equation \eqref{evsys1}, we conclude that there is a sequence $(z_n)_{n\ge-1}$ that solves the difference equation  \eqref{main} (with no boundary condition) such that 
 \begin{equation}\label{Y+z}
 Y^+_n=\begin{bmatrix}0&z_n\\0&z_{n-1}\end{bmatrix}, \, n\ge0.
 \end{equation}
 \end{remark}
 
 Next, we proceed to discuss the Jost solution $\bz^+=(z_n^+)_{n\ge-1}$  obtained as a solution to  the following scalar Volterra equation,
\begin{equation}\label{jostz2-new}
z_n^+-(\mu_+)^n=-(\mu_+-\mu_-)^{-1}\sum_{k=n}^{+\infty}b_kc_k\big((\mu_+)^{n-k}
-(\mu_-)^{n-k}\big)z^+_k,
\end{equation}
at first for $n\ge N$ with $N$ sufficiently large. 
A computation using $(\mu_+)^2+\lambda\mu_+-1=0$ shows that $z_n^+$ from \eqref{jostz2-new} satisfy \eqref{main}, see \cite{LG21} for a similar compuation.
Therefore,  the solution $(z_n^+)_{n\ge N}$ to \eqref{jostz2-new} 
can be propagated backward as solution to \eqref{main}, and thus $z^+_n$ can be defined for all $n\ge-1$.
We now record properties of the Jost solution.

\begin{remark}\label{holF} The Jost solution $\bz^+(\lambda)$ is unique and holomorphic in $\lambda$. This follows by a standard argument, cf.\ \cite{CL07,GLM07} and the references therein, presented in \cite[Remark 2.2]{LV24} and based on passing in \eqref{jostz2-new} to the new unknowns $z^+_n/\mu_+^n$ and proving that the RHS of the resulting equation is a strict contraction in the space of exponentially decaying at infinity sequences, see the analogous  proof of Lemma \ref{lem:Volterra}. As in the lemma, this yields the property  \eqref{zpm-asymp}. 

We stress that the Jost solution $\bz^+=(z^+_n(\lambda))_{n\ge-1}$ is defined for $n\ge-1$ as the solution of the difference equation \eqref{main}, with no boundary condition \eqref{main-bc}. Theorem \ref{thm:main} shows, in particular,  that  $z_{-1}^+(\lambda)=0$ if and only if $\lambda$ is an eigenvalue (which is not surprising as in this case $\bz^+(\lambda)$ is the exponentially decaying solution to \eqref{main} that also satisfies the boundary condition \eqref{main-bc}).\end{remark}
\begin{remark}\label{rem:lem2.3} Let us consider the difference equation \eqref{eulerev1} which is a particular case of \eqref{main} with $b_n,c_n$ as in \eqref{defbc} and $\rho_n$
satisfying \eqref{condrho}. Then the proof of \cite[Lemma 2.3]{LV24} applies and shows that the Jost solution satisfies $z_n^+\neq 0$ for all $n\ge0$. 
 \end{remark}

 \begin{proof}[Proof of Theorem \ref{thm:main}]
 The three equalities in \eqref{bigfive} are proved as follows.

 {\bf 1.} We prove $\det(I-K_\lambda^+)=\det(I-\cT_\lambda^+)$.\, We claim that $K_\lambda^+$ from \eqref{defK} is the $(1,1)$-block of the operator $\cT_\lambda^+$ from \eqref{defT} in the decomposition \[\ell^2(\bbZ_+;\bbC^2)=\ran(\diag\{Q_+\})\oplus\ran(\diag\{Q_-\}),\]  with $\diag=\diag_{\bbZ_+}$, for the projections $Q_\pm$ from \eqref{Anot}, that is, that 
\begin{equation}\label{clKT}
\diag\{C_n\}(S^{-1}-A)^{-1}\diag\{B_n\}=\begin{bmatrix}-\diag\{c_n\}(S-S^{-1}-\lambda)^{-1}\diag\{b_n\}&0\\0&0\end{bmatrix}.
\end{equation}
Clearly, this implies the required equality. 

To begin the proof of the claim, we fix $\bu=(u_n)_{n\ge0}\in\ell^2(\bbZ_+;\bbC^2)$ amd denote $v_n=\big(\diag_{n\in\bbZ_+}\{Q_+\}(S^{-1}-A)
\diag_{n\in\bbZ_+}\{Q_+\}\bu\big)_n$. Formulas \eqref{Anot} and \eqref{formRpm} show that  $AR_+=\mu_+R_+$ and $R_-A^{-1}=-\mu_+R_-$, and thus \eqref{invform} and \eqref{AA} imply
that $v_n$ is equal to
\begin{equation}\label{vn}
-\sum_{k=n}^{+\infty}(-\mu_+)^{k+1}Q_+W\widetilde{A}^nW^{-1}R_-Q_+u_k+\sum_{k=0}^{n-1}\mu_+^nQ_+R_+W\widetilde{A}^{-(k+1)}W^{-1}Q_+u_k.
\end{equation}
We now use formulas \eqref{winv1} and \eqref{formRpm} to compute the matrices
$Q_+W$, $W^{-1}R_-Q_+$, $Q_+R_+W$, $W^{-1}Q_+$. Plugging this into \eqref{vn} and using that $\mu_\pm$ solve the equation $\mu^2+\lambda\mu-1=0$, after a tedious computation we conclude that the first component of the vector $v_n=Q_+v_n\in\bbC^2$ is given by the formula
\[(\mu_+-\mu_-)^{-1}\Big(\sum_{k=n}^{+\infty}(\mu_-)^{n-k}\big(1-\big(\frac{\mu_+}{\mu_-}\big)^{n+1}\big)z_k+\sum_{k=0}^{n-1}(\mu_+)^{n-k}\big(1-\big(\frac{\mu_+}{\mu_-}\big)^{k+1}\big)z_k\Big).\]
We now use \eqref{INV0} to recognize that the last expression is $((S-S^{-1}-\lambda)^{-1}\bz)_n$ where $z_n$ denote the first component of the vector $u_n\in\bbC^2$, $n\ge0$. Multiplying by $c_n$ and $b_n$ and recalling \eqref{Anot} yields
the required claim \eqref{clKT}.

{\bf 2.} We prove $\det(I-\cT_\lambda^+)=\cE^+(\lambda)$ in three steps. 
The first step is to show that it suffices to prove the equality only for finitely supported $(B_n)_{n\ge0}$ and $(C_n)_{n\ge0}$. This is fairly standard and follows, say, 
as the {\em claim} in the proof of  \cite[Theorem 4.6]{CL07} or the proof of \cite[Theorem 4.3]{GM04}: Indeed, replace $B_k,C_k$ in \eqref{defTN} by $\mathbf{1}_M(k)B_k$, $\mathbf{1}_M(k)C_k$ where $\mathbf{1}_M$
is the characteristic function of $[0,M]$ equals to one on the segment and to zero outside.
As in the proof of Lemma \ref{lem:Volterra}, we denote the respective operator by $T_N^{(M)}$ and the respective matrix-valued Jost solution by $\bY^{+,M}$. Choosing a large $N$ and even larger $M$ the proof of Lemma \ref{lem:Volterra} yields $\|T_N-T_N^{(M)}\|_{\cB(\ell^\infty_N)}\lesssim q_{M+1}\to0$ as $M\to\infty$ for the operator norm on $\ell^\infty_N$. Then $\|Y^+_0-Y^{+,M}_0\|_{\bbC^{2\times 2}}\to0$ as in the proof of Lemma \ref{lem:Volterra} and therefore $\cE^{+,M}(\lambda)\to\cE^+(\lambda)$ as $M\to\infty$ for the Evans function $\cE^{+,M}$ obtained by replacing $Y^+_0$ in \eqref{defcE} by $Y_0^{+,M}$. Analogously, let $\cT_\lambda^{+,M}$ denote the operator in \eqref{defT} obtained by replacing  
$B_k,C_k$ by $\mathbf{1}_M(k)B_k$, $\mathbf{1}_M(k)C_k$. Then \eqref{condbc} yields convergence of $\cT_\lambda^{+,M}$ to $\cT_\lambda^+$ in the trace class norm and so
$\det(I-\cT_\lambda^{+,M})\to\det(I-\cT_\lambda^+)$ as $M\to\infty$ completing the first step in the proof.
From now on we thus assume that sequences $(B_n)_{n\ge0}, (C_n)_{n\ge0}$ are finitely supported.

The second step in the proof is a reduction of the infinite dimensional determinant $\det(I-\cT_\lambda^+)$ to a finite dimensional one. It follows a well established path, see \cite{GLM07,GM04,GN15,GGK90}. Multiplying the operator in the RHS of  \eqref{invform} by 
$\diag_{n\in\bbZ_+}\{C_n\}$ and $\diag_{k\in\bbZ_+}\{B_k\}$ 
and adding and subtracting $\sum_{k=n}^{+\infty}C_nA^nR_-A^{-(k+1)}B_ku_k$ to the result we arrive at the identity 
\begin{equation}\label{TH0}
\cT_\lambda^+=H_++H_2H_3, \text{ where $H_+=H_0+H_1$,}\end{equation} and we introduce notations
\begin{eqnarray}\label{TH} 
 H_0&=&-\diag_{n\in\bbZ_+}\{C_nA^{-1}B_n\},\quad
(H_1\bu)_n=-\sum_{k=n+1}^{+\infty}C_nA^{n-(k+1)}B_ku_k,\\
(H_2y)_n&=&C_nA^nR_+y, \quad H_3\bu=R_+\sum_{k=0}^{+\infty}A^{-(k+1)}B_ku_k
\end{eqnarray}
for $y\in\bbC^2$ and $\bu=(u_n)_{n\ge0}\in\ell^2(\bbZ_+;\bbC^2)$. Here, the operators $H_0$ and $H_1$ in $\ell^2(\bbZ_+;\bbC^2)$ are of trace class by \eqref{condbc} while the operators $H_2$ and $H_3$ are of rank one as $H_2$ acts from $\ran R_+=\ran P_+\subset\bbC^2$ into $\ell^2(\bbZ_+;\bbC^2)$ while $H_3$ acts from  $\ell^2(\bbZ_+;\bbC^2)$ into $\ran R_+=\ran P_+\subset\bbC^2$. The diagonal operator $I-H_0$ is invertible with $\det(I-H_0)$ being equal to $1$  by \eqref{Anot1} because
\[\prod_{n=0}^{+\infty}\det(I+C_nA^{-1}B_n)=\prod_{n=0}^{+\infty}\det(I+A^{-1}B_nC_n)=\prod_{n=0}^{+\infty}\det(A^{-1})\det(A+B_nC_n).\]
 The operator $H_1$ is block-lower-triangular, and thus $I-H_+$ is invertible
with $\det(I-H_+)=1$. Writing $I-\cT_\lambda^+=(I-H_+)\big(I-(I-H_+)^{-1}H_2H_3\big)$ and changing the order of factors in the determinant, we arrive at the identity
\begin{equation}\label{idTH}
\det(I_{\ell^2(\bbZ_+;\bbC^2)}-\cT_\lambda^+)=
\det_{\bbC^{2\times 2}}\big(I_{2\times 2}-H_3(I-H_+)^{-1}H_2\big)
\end{equation}
that reduces the computation of the infinite dimensional determinant for $I-\cT_\lambda^+$ to the finite dimensional determinant for the operator $H_3(I-H_+)^{-1}H_2$ acting in $\ran R_+$ thus completing the second step in the proof.

The third step relates the RHS of \eqref{idTH} and the solution $(Y_n^+)_{n\ge0}$ to the Volterra equation 
\eqref{VE}. Indeed, for any $y\in\bbC^2$ we let $u_n=C_nY_n^+y$ and $\bu=(u_n)_{n\ge0}$. Multiplying \eqref{VE} from the left by $C_n$ and applying the resulting matrices to $y$ yields 
\[u_n=C_nA^nR_+y-\sum_{k=n}^{+\infty}C_nA^{n-(k+1)}B_ku_k=(H_2y)_n+
(H_+\bu)_n, \, n\in\bbZ_+,\]
or $\bu=(I-H_+)^{-1}H_2y$. Multiplying by $H_3$ from the left, by the definition of $H_3$,
\[H_3(I-H_+)^{-1}H_2y=H_3\bu=R_+\sum_{k=0}^{+\infty}A^{-(k+1)}B_ku_k.\]
Since $y$ is arbitrary, the respective matrices are equal, that is,
\[H_3(I-H_+)^{-1}H_2=R_+\sum_{k=0}^{+\infty}A^{-(k+1)}B_kC_kY_k^+.\]
 Since $Y^+_n=Y^+_nR_+$ by \eqref{mvj}, the last identity and formula \eqref{idTH} yield
\begin{equation}\label{RR}
\det(I-\cT_\lambda^+)=\det_{\bbC^{2\times 2}}\big(I_{2\times 2}-R_+\sum_{k=0}^{+\infty}A^{-(k+1)}B_kC_kY^+_kR_+\big).
\end{equation}
Writing matrices in the block form using the 
direct sum decomposition $\bbC^2=\ran R_+\oplus\ran R_-$ gives the equality of the determinants of the following two matrices,
\begin{eqnarray*}
I_{2\times 2}-\sum_{k=0}^{+\infty}A^{-(k+1)}B_kC_kY^+_kR_+&=&\begin{bmatrix}
R_+-R_+\sum_{k=0}^{+\infty}A^{-(k+1)}B_kC_kY^+_kR_+&0\\
-R_-\sum_{k=0}^{+\infty}A^{-(k+1)}B_kC_kY^+_kR_+&R_-\end{bmatrix},\\
I_{2\times 2}-R_+\sum_{k=0}^{+\infty}A^{-(k+1)}B_kC_kY^+_kR_+&=&
\begin{bmatrix}
R_+-R_+\sum_{k=0}^{+\infty}A^{-(k+1)}B_kC_kY^+_kR_+&0\\
0&R_-\end{bmatrix}.
\end{eqnarray*}
This, \eqref{VE} with $n=0$, and \eqref{RR} show that  $\det(I-\cT_\lambda^+)$ is equal to the determinant of the matrix 
\[I_{2\times 2}-\sum_{k=0}^{+\infty}A^{-(k+1)}B_kC_kY^+_kR_+=R_+-\sum_{k=0}^{+\infty}A^{-(k+1)}B_kC_kY^+_k+R_-=Y^+_0+R_-,\]
completing the proof of the equality $\det(I-\cT_\lambda^+)=\cE^+(\lambda)$ by \eqref{defcE}.

{\bf 3.}\, We prove $\cE^+(\lambda)=\cF^+(\lambda)$. Using the Jost solution $(z^+_n)_{n\ge-1}$ and the regular solution satisfying $z_0^{\rm r}=1$, we recall that $\cW(\bz^+,\bz^{\rm r})_0=z^+_{-1}$ in \eqref{defF}. By Remark \ref{rem:Y+} the solution $(Y_n^+)_{n\ge0}$ is of the form \eqref{Y+z} with some $(z_n)_{n\ge-1}$. We claim that the solution $(z_n)_{n\ge-1}$ in formula \eqref{Y+z} for $Y_n^+$ satisfies $z_n=\mu_+z_n^+$ where $(z^+_n)_{n\ge-1}$ is the Jost solution from \eqref{jostz2-new}. Assuming the claim, we use formulas \eqref{defcE}, \eqref{Y+z}, \eqref{formRpm}, and \eqref{defF} to obtain the desired result,
\[
\cE^+(\lambda)=\det(Y^+_0+R_-)=\det\begin{bmatrix}1\,&\,\,-\mu_++\mu_+z^+_0\\0\,&\,\,\mu_+z^+_{-1}\end{bmatrix}=\mu_+(\lambda)z^+_{-1}(\lambda)=\cF^+(\lambda).
\]

It remains to prove the claim  $z_n=\mu_+z_n^+$ in formula \eqref{Y+z}. Using \eqref{winv1} we multiply equation \eqref{VE} by $W^{-1}$ from the left and $W$ from the right passing to the new unknowns $\widetilde{Y}^+_n$ and using \eqref{winv1} and  $(z_n)_{n\ge-1}$ from \eqref{Y+z},
\[\widetilde{Y}^+_n=W^{-1}Y_n^+W=(\mu_+-\mu_-)^{-1}\begin{bmatrix}z_n-\mu_-z_{n-1}&z_n-\mu_-z_{n-1}\\-(z_n-\mu_+z_{n-1})\,\,&\,\,-(z_n-\mu_+z_{n-1})\end{bmatrix}.\]
 Noting that $A^nR_+=\mu_+^nR_+$ by \eqref{Anot} and \eqref{formRpm}, and explicitly computing all matrices that appear in the equation for $\widetilde{Y}^+_n$, we arrive at the following equations,
\begin{eqnarray*}
(\mu_+-\mu_-)^{-1}\big(z_n-\mu_-z_{n-1}\big)&=&\mu_+^n-(\mu_+-\mu_-)^{-1}
\sum_{k=n}^{+\infty}b_kc_k\mu_+^{n-(k+1)}z_k,\\
(\mu_+-\mu_-)^{-1}\big(-z_n+\mu_+z_{n-1}\big)&=&-(\mu_+-\mu_-)^{-1}
\sum_{k=n}^{+\infty}b_kc_k\big(-\mu_-^{n-(k+1)}\big)z_k.
\end{eqnarray*}
Multiplying the first equation by $\mu_+$ and the second by $\mu_-$ and adding the results,
\[z_n=\mu_+^{n+1}-(\mu_+-\mu_-)^{-1}\sum_{k=n}^{+\infty}b_kc_k\big(\mu_+^{n-k}-\mu_-^{n-k}\big)z_k.\]
Divided by $\mu_+$, this is equation \eqref{jostz2-new} for $z_n/\mu_+$ and the uniqueness of the solution of the equation yields the claim thus completing the proof of \eqref{bigfive} in the theorem.

{\bf 4.}\, To prove the last assertion in the theorem, we rely on the Birman-Schwinger principle saying that $\lambda\in\spec\big(S-S^{-1}+\diag_{n\in\bbZ_+}\{b_nc_n\}\big)$ if and only if $\lambda$ is a zero of $\det(I-K_\lambda^+)$ with the same multiplicity which follows from the fact that the operator 
$S-S^{-1}+\diag_{n\in\bbZ_+}\{b_nc_n\}-\lambda$ can be written as the product
\[\big(S-S^{-1}-\lambda\big)\big(I_{\ell^2(\bbZ_+;\bbC)}-\big(S-S^{-1}-\lambda\big)^{-1}\diag_{n\in\bbZ_+}\{b_n\}\times\diag_{n\in\bbZ_+}\{c_n\}\big)\]
and the standard property $\det(I-K_1\times K_2)=\det(I-K_2\times K_1)$ of the operator determinants. That the eigenvalues $\lambda$ are simple follows from the fact that $\bz$ is an eigensequence if and only if $\bz$ is proportional to $\bz^+(\lambda)$ with $\lambda$ satisfying $z_{-1}^+(\lambda)=0$ (in other words, the dichotomy projection for \eqref{evsys1} has rank one).\end{proof}
 \begin{proof}[Proof of Corollary \ref{cor:main}]
We follow \cite{LV24} and show that $\cF^+(\lambda)=\cG^+(\lambda)$ for $\cF^+$ from \eqref{defF} and $\cG^+$ from \eqref{defG}.
 Here, we assume that $\re(\lambda)>0$.  
The function $\cF^+$ is holomorphic in $\lambda$ by Remark \ref{holF}. The convergent continued fraction in \eqref{deffg} is holomorphic in $\lambda$ by the classical Stieltjes-Vitali Theorem \cite[Theorem 4.30]{JT}. Thus, it is enough to show the desired equality only for $\lambda>0$ which we assume from now on.
 We recall that $z_n^+\neq0$ for $n\ge0$ by Remark \ref{rem:lem2.3}, introduce $v^\pm_n=z^\pm_{n-1}/z^\pm_n$ and re-write equation 
 \eqref{eulerev1} for $\bz^\pm$ as $v^+_n=\frac{\lambda}{\rho_n}+\frac{1}{v^+_{n+1}}$ for $n\ge0$.
 The last formula iterated forward produces the continued fraction \eqref{deffg}.
 Because the continued fraction converges and $\lambda>0$, an argument from \cite[pp.2063]{DLMVW20} involving monotonicity of the sequences formed by the odd and even truncated continued fractions yields 
$ v_0^+=g_+(\lambda)+\lambda/\rho_0$. Therefore,
 \[\cW(\bz^+,\bz^{\rm r})_0=z^+_{-1}=z^+_0v^+_0=z^+_0(g_+(\lambda)+\lambda/\rho_0),\]
 and the desired equality of $\cF^+$ and $\cG^+$ follows by \eqref{defF} and \eqref{defG}.
\end{proof}


\begin{lemma}\label{lem:hl} Assume \eqref{condpq} and $\lambda>0$. Then the solutions of \eqref{weq} on $\bbZ$ are in one-to-one correspondence with the $\ell^2(\bbZ_+;\bbC)$-solutions to the eigenvalue problem \eqref{eulerev1} given by $z_n=\rho_nw_n$ for $n\ge-1$ and $w_n=0$ for $n\le-2$.
\end{lemma}
\begin{proof}
Suppose that $(w_{n})_{n\in\bbZ}$ solves \eqref{weq} and let $z_n=\rho_nw_n$ for all $n\in\bbZ$. Then $z_{-1}=0$ since $\rho_{-1}=0$ and the sequence  $(z_{n})_{n\le-1}$ solves the equation 
\begin{equation}\label{zneq}
z_{n-1}-z_{n+1}=\lambda z_n/\rho_{n} \text{ for $n\le-2$ and $z_{-1}=0$.}
\end{equation}
If $z_{-2} \neq 0$ then the sequence $(z_{n})_{n \leq -2}$ due to $\lambda>0$ exponentially grows as $n\to-\infty$ as shown in  the proof of \cite[Lemma 2.3]{LV24} contradicting $(w_n)_{n\in\bbZ}\in\ell^2(\bbZ;\bbC)$ in \eqref{weq}. So one has $z_{-2}=0$, $z_{-1}=0$ and, as a result, $z_{n}=0$ for all $n \leq -2$. This, in turn, implies that $w_{n}=0$ for all $n \leq -2$. 


Conversely, suppose that $(z_{n})_{n\ge0} \in \ell_{2}(\mathbb Z_{+})$, $z_{-1}=0$ solves \eqref{eulerev1} and let $w_{n}= z_{n}/\rho_{n}$ for $n\ge0$, $w_{-1}= -z_{0}/\lambda$, and $w_n=0$ for $n\le-2$. Then $(w_{n})_{n\in\bbZ}$ so defined solves \eqref{weq}. 
\end{proof}

 \begin{proof}[Proof of Corollary \ref{cor:main:eul}]
 We follow the proof of \cite[Corollary 1.3]{LV24}. Due to Lemma \ref{lem:hl}, Theorem \ref{thm:main} yields the first assertion in the corollary. To finish the proof it suffices  to show that there is a positive root of the function $\cG^+$ defined in \eqref{defG}, equivalently, that  for some $\lambda>0$ and $g_+$  defined in \eqref{deffg} one has  $-\lambda/\rho_0=g_+(\lambda)$. We recall that $\rho_0<0$ and $g_+(\lambda)>0$ for $\lambda>0$ because condition \eqref{condrho} holds. The proof is completed by using the relations
 $\lim_{\lambda\to0^+}g_+(\lambda)=1$ and $\lim_{\lambda\to+\infty}g_+(\lambda)=0$ established in \cite[Lemma 2.10(4)]{DLMVW20}.
\end{proof}


\end{document}